\documentclass{amsart}
\usepackage{amsmath,amssymb,amsthm,mathtools}
\usepackage{enumitem}
\usepackage{hyperref}

\usepackage{nicematrix}
\usepackage{xcolor}

\newtheorem{theorem}{Theorem}[section]
\newtheorem{lemma}[theorem]{Lemma}
\newtheorem{proposition}[theorem]{Proposition}
\newtheorem{corollary}[theorem]{Corollary}
\numberwithin{equation}{section}

\theoremstyle{definition}
\newtheorem{definition}[theorem]{Definition}
\newtheorem{example}[theorem]{Example}

\theoremstyle{remark}
\newtheorem{remark}[theorem]{Remark}

\newcommand{\Z}{\mathbb Z}
\newcommand{\C}{\mathbb C}
\newcommand{\R}{\mathbb R}
\newcommand{\Q}{\mathbb Q}

\newcommand{\cC}{\mathcal C}

\newcommand{\cM}{\mathcal M}
\newcommand{\fF}{\mathfrak F}
\newcommand{\RZ}{\mathbb R\mathcal Z}
\newcommand{\bj}{\mathbf j}

\newcommand{\one}{\mathbf 1}
\newcommand{\field}{\Bbbk}
\newcommand{\genus}{\mathfrak g}
\newcommand{\etale}{\'etale }

\DeclareMathOperator{\supp}{supp}
\DeclareMathOperator{\row}{row}
\DeclareMathOperator{\im}{im}
\DeclareMathOperator{\Hom}{Hom}

\title{Symplectic and projective small covers over products of polygons}

\author[S. Choi]{Suyoung Choi}
\address{Department of Mathematics, Ajou University, 206, World cup-ro, Yeongtong-gu, Suwon 16499, Republic of Korea}
\email{schoi@ajou.ac.kr}

\date{\today}
\thanks{This work was supported by the National Research Foundation of Korea Grant funded by the Korean Government (RS-2025-00521982).}



\begin{document}

\begin{abstract}
We study symplectic and projective structures on small covers over products of polygons.
We introduce the factor-compatible class for small covers over products of polygons and prove that every factor-compatible small cover admits a smooth projective model as a finite quotient of a product of curves.
Furthermore, we show that the graded mod~$2$ cohomology ring determines the Hodge diamond of the associated projective model.
We also prove that every factor-compatible small cover admits an iterated equivariant bundle structure.
\end{abstract}
\maketitle


\section{Introduction}
A small cover over an $n$-dimensional simple polytope $P$, introduced by Davis and Januszkiewicz~\cite{Davis-Januszkiewicz1991}, is a closed smooth $n$-manifold equipped with a locally standard action of an elementary abelian $2$-group of rank $n$ whose orbit space is $P$.
It may be regarded as a real analogue of a quasitoric manifold. 
More concretely, if $P$ is an $n$-dimensional simple polytope and $\lambda$ is a mod~$2$ characteristic map over~$P$, then the corresponding small cover $M(P,\lambda)$ is obtained as a free quotient of the real moment-angle manifold. 
Thus the topology of a small cover is encoded by the combinatorics of $P$ together with the linear data $\lambda$.

The purpose of this paper is to study when small covers over products of polygons admit symplectic, K\"ahler, or projective structures, focusing on a natural class defined by factorwise compatibility conditions.
A symplectic manifold must be even-dimensional and cohomologically symplectic, but for small covers these conditions are quite restrictive. 
One of the most important examples is given by real Bott manifolds. 
A real Bott manifold is a small cover over a cube. 
Ishida~\cite{Ishida2011} proved that, within the real Bott class, the cohomologically symplectic, symplectic, and K\"ahler conditions are equivalent. 
In particular, symplectic real Bott manifolds form a distinguished subclass of small covers over cubes.

An even-dimensional real Bott manifold is a small cover over an even-dimensional cube. 
Equivalently, its orbit space can be viewed as a product of squares. 
This suggests a natural direction of generalization: replace the square factors by arbitrary polygons and study small covers over products of polygons
$$
    P=P_{m_1}\times\cdots\times P_{m_n}.
$$
The real moment-angle manifold of $P$ is a product of closed surfaces. 
Hence the problem becomes to understand when the quotient of such a product of surfaces by a free elementary abelian $2$-group action admits a symplectic, K\"ahler, or projective structure.

The main notion introduced in this paper is that of a factor-compatible small cover over a product of polygons. 
Roughly speaking, the condition separates the polygon factors into two types. 
For a non-square factor, it requires the quotient action to preserve the orientation of the corresponding surface factor. 
For the square factors, it requires the opposite-pair weights to be paired in a way compatible with a complex torus structure. 
When all factors are squares, this condition recovers Ishida's symplectic condition for real Bott manifolds.
Thus factor-compatibility may be viewed as a polygonal analogue of the symplectic real Bott condition.

Our first main result is that factor-compatibility is sufficient for the existence of strong geometric structures. 
More precisely, we prove that every factor-compatible small cover over a product of polygons admits a smooth projective model. 
The model is obtained as a finite free quotient
$$
    X=(\cC_1\times\cdots\times \cC_n)/G
$$
of a product of smooth projective curves. 
In particular, every factor-compatible small cover is K\"ahler, symplectic, and cohomologically symplectic. 
This is proved in Theorem~\ref{thm:factor-compatible-projective}.

In dimension four, the converse is known for small covers over products of two polygons~\cite{Choi2026}. 
These results suggest that factor-compatibility captures a substantial and geometrically natural part of the symplectic range in this setting, although a complete converse is not known in higher dimensions.
At present, no symplectic small cover over a product of polygons is known to us that is not factor-compatible. 
In this paper we record some elementary obstructions. 
For example, the presence of a triangle factor prevents any small cover over the product from being cohomologically symplectic. 
We also show that if all polygon factors have an odd number of sides, then no small cover over the product is orientable.

Factor-compatibility is not merely a sufficient condition for symplecticity.
It is a condition that simultaneously produces a projective quotient model, makes the Hodge theory computable, and forces a tower structure.

For a factor-compatible small cover, let $X$ be its smooth projective model. 
We prove that the Hodge polynomial of $X$ is determined by certain character-theoretic invariants of the $G$-action on the curve factors.
A consequence is that the rational Betti numbers of $X$ determine its full Hodge diamond. 
Combining this with the mod~$2$ cohomology ring of the underlying small cover, we show that the graded mod~$2$ cohomology ring determines the Hodge diamond of the associated projective model; see Theorem~\ref{thm:mod2-rigidity-factor-compatible}. 
This result should be understood not as a full cohomological rigidity statement of the diffeomorphism type, but as a rigidity phenomenon for the Hodge-theoretic information carried by the projective models.

Finally, we prove a structural result for factor-compatible small covers. 
After a natural regrouping of the square factors, the characteristic matrix of a factor-compatible small cover can be put into block lower triangular form, up to Davis--Januszkiewicz equivalence. 
As a consequence, every factor-compatible small cover admits an iterated equivariant bundle structure whose successive fibers are orientable small covers over polygons; see Corollary~\ref{cor:iterated-equivariant-bundle}.
This gives a surface-type analogue of the iterated bundle structure of generalized real Bott towers.

The paper is organized as follows. 
In Section~\ref{sec-preliminaries} we recall real moment-angle manifolds, small covers, and the cohomological formulas used throughout the paper. 
In Section~\ref{sec-products} we introduce factor-compatible small covers and construct their smooth projective models. 
Section~\ref{sec-symplectic-obstructions} contains elementary obstructions to symplecticity. 
In Section~\ref{sec-hodge-mod2} we compute the Hodge polynomial of the projective models and prove the mod $2$ cohomological determination of the Hodge diamond.
In Section~\ref{sec-block-triangular} we prove the block triangular form and the resulting iterated equivariant bundle structure.

\section{Preliminaries}\label{sec-preliminaries}

We begin by fixing notation and recalling the facts used in the paper.
We write
$$
    \Z_2=\Z/2\Z=\{0,1\}.
$$
All vector spaces over $\Z_2$ are written additively.
For a ring $\field$ and a topological space $X$, we denote by $H^i(X;\field)$ its $i$th cohomology group with coefficients in~$\field$, and 
put $b_i(X;\field)=\dim_\field H^i(X;\field)$.
When $\field=\Q$, we write $b_i(X)$.

\subsection{Real moment-angle manifolds and small covers}

Let $P$ be a simple polytope of dimension~$n$ with the facet set $\fF(P)=\{F_1,\ldots,F_m\}$.
For a point $p \in P$, let $F(p)$ denote the unique face whose relative interior contains $p$.

Let $e_1,\ldots,e_m$ be the standard basis of $\Z_2^m$, indexed by the facets of~$P$.
The \emph{real moment-angle manifold}~$\RZ_P$ of~$P$ is defined by
$$
    \RZ_P=(\Z_2^m\times P)/\sim_P,
$$
where $(g,p)\sim_P(h,p)$ if and only if $g+h\in \langle e_i\mid F(p)\subset F_i\rangle$.
The group $\Z_2^m$ acts on~$\RZ_P$ by $h\cdot[g,p]=[h+g,p]$ for $h \in \Z_2^m$.
The orbit space of this action is naturally identified with~$P$.

A \emph{mod~$2$ characteristic map} over $P$ is a map $\lambda\colon \Z_2^m \to \Z_2^n$ satisfying the following nonsingularity condition.
\begin{quote}
    Whenever $F_{i_1}\cap\cdots\cap F_{i_k}\neq\varnothing$, the vectors $\lambda(e_{i_1}),\ldots,\lambda(e_{i_k})$ are linearly independent over $\Z_2$.
\end{quote}
We denote its row space by $\row(\lambda)\subset \Z_2^m$.

Let $\lambda$ be a characteristic map over $P$.
Then $G=\ker\lambda \subset \Z_2^m$ acts freely on~$\RZ_P$; see \cite[Lemma~3.1]{Choi-Kaji-Theriault2017}.
The \emph{small cover} determined by $(P,\lambda)$ is 
$$
    M(P,\lambda) = \RZ_P/G.  
$$
Since $\lambda$ is surjective, we have $\Z_2^n \cong \Z_2^m / G$.
Thus $\Z_2^n$ acts naturally on $M(P,\lambda)$, and the orbit space is again~$P$.

Conversely, every small cover over $P$ arises from a characteristic map over $P$ in this way~\cite{Davis-Januszkiewicz1991}.
Hence a small cover over $P$ may be viewed as a free finite quotient of~$\RZ_P$.
 
Two small covers over the same polytope $P$ are called \emph{Davis--Januszkiewicz equivalent}, or \emph{D--J equivalent}, if there is a weakly $\Z_2^n$-equivariant homeomorphism between them covering the identity map on~$P$.
Changing the basis of $\Z_2^n$ corresponds to applying elementary row operations to the matrix~$\lambda$.
Therefore elementary row operations do not change the small cover up to D--J equivalence.

Let $v_i\in H^1(M(P,\lambda);\Z_2)$ be the degree-one class corresponding to the facet $F_i$.
The Davis--Januszkiewicz formula~\cite{Davis-Januszkiewicz1991} gives 
\begin{equation}\label{eq:mod2cohom_smallcover}
    H^\ast(M(P,\lambda);\Z_2) \cong \Z_2[v_1,\ldots,v_m]/(I_P+J_\lambda).
\end{equation}
Here $I_P$ is the Stanley--Reisner ideal of $P$.
It is generated by the square-free monomials $v_{i_1}\cdots v_{i_r}$ such that $F_{i_1}\cap\cdots\cap F_{i_r}=\varnothing$.
The ideal $J_\lambda$ is generated by the linear relations
$$
    \sum_{i=1}^m \lambda_{ji}v_i, \qquad j=1,\ldots,n.
$$

For a vector $\omega = (\omega_1, \ldots, \omega_m) \in \Z_2^m$, put $\supp(\omega) = \{i \mid \omega_i =1\} \subset \{ 1, \ldots, m\}$.
We define 
$$
    P_\omega = \bigcup_{i \in \supp(\omega)} F_i.
$$
When $\omega = 0$, this means $P_\omega = \varnothing$.
When $\omega=\one_m = (1,\ldots,1)$, this means $P_\omega = \partial P$.
The rational cohomology groups of $\RZ_P$ and $M(P,\lambda)$ are deeply related to the reduced cohomology groups of $P_\omega$; see~\cite{Choi-Park2020}.

The $\Z_2^m$-equivariant Hochster decomposition for real moment-angle manifolds is
\begin{equation}\label{eq:rational_additive_formula_RZP}
    H^q(\RZ_P;\Q) \cong \bigoplus_{\omega\in \Z_2^m} \widetilde H^{q-1}(P_\omega;\Q) \otimes \Q_\omega.
\end{equation}
where $\Q_\omega$ is the one-dimensional sign representation of character $\omega$.
Here, we use the convention $\widetilde{H}^{-1}(\emptyset;\Q)=\Q$ and $\widetilde{H}^{-1}(X;\Q)=0$ for nonempty $X$.
Furthermore, for a small cover $M(P,\lambda)$, it is shown in~\cite{Choi-Park2017Cohomology} that
\begin{equation}\label{eq:rational_additive_formula}
    H^q(M(P,\lambda);\Q) \cong H^q(\RZ_P;\Q)^G \cong \bigoplus_{\omega\in \row(\lambda)} \widetilde H^{q-1}(P_\omega;\Q).
\end{equation}
This is the additive form of the rational cohomology formula.
The ring-level version appears in~\cite{Choi-Park2020}.
From \eqref{eq:rational_additive_formula}, we immediately have the orientability criterion for small covers.
\begin{corollary}[\cite{Nakayama-Nishimura2005}]\label{cor:NN}
Let $M(P,\lambda)$ be a small cover.
Then $M(P,\lambda)$ is orientable if and only if
$$
    \one_m\in \row(\lambda).
$$
\end{corollary}

The integral cohomology group of a small cover can also be recovered from the data of $P_\omega$ by Cai--Choi~\cite{Cai-Choi2021}.
From the formula, we have the following corollary.
\begin{corollary}[{\cite[Corollary~1.5]{Cai-Choi2021}}] \label{thm:Cai-Choi-two-torsion}
Let $M=M(P,\lambda)$ be a small cover.
Assume that $H^\ast(P_\omega;\Z)$ is torsion-free for every $\omega\in \row(\lambda)$.
Then every torsion element of~$H^\ast(M;\Z)$ has order~$2$.
\end{corollary}

We record a simple duality statement for row spaces for later use.
\begin{lemma}\label{lem:row-kernel-duality}
Let $\lambda\colon \Z_2^m\longrightarrow \Z_2^n$ be a linear map, and let $G=\ker\lambda$.
For $u\in \Z_2^m$, the following two conditions are equivalent.
\begin{enumerate}
\item $u\notin \row(\lambda)$.
\item There exists $\tau\in G$ such that $\langle u,\tau\rangle=1$, where $\langle-,-\rangle$ denotes the standard pairing on $\Z_2^m$.
\end{enumerate}
\end{lemma}

\begin{proof}
The standard pairing gives
$$
    \row(\lambda)=(\ker\lambda)^\perp.
$$
The claim follows at once.
\end{proof}

\subsection{Small covers over products of polygons}

Denote by $P_m$ the polygon with $m$ sides.
For a single $m$-gon $P_m$, the manifold $\RZ_{P_m}$ is a closed orientable surface.
Its standard cell decomposition has
$$
    f_0=2^m, \quad f_1=m2^{m-1},\quad \text{ and } \quad f_2=m2^{m-2}.
$$
The Euler characteristic is $\chi(\RZ_{P_m})=2^{m-2}(4-m)$, and, hence, its genus is
$$
    \genus(m):=1+(m-4)2^{m-3}.
$$
We write $\RZ_{P_m}\cong \Sigma_{g(m)}$.
In particular, $\RZ_{P_3}$ is a sphere~$S^2$ and $\RZ_{P_4}$ is a torus~$T^2$.

We prepare one linear algebraic lemma for later computation.
\begin{lemma}\label{lem:orientation-character-polygon}
    For the canonical $\Z_2^m$-action on $\RZ_{P_m}$, an element $\tau\in\Z_2^m$ preserves the orientation of $\RZ_{P_m}$ if and only if 
    $$
        \langle \one_m,\tau\rangle=0.
    $$
    Equivalently, the orientation character of this action is $\one_m$.
\end{lemma}

\begin{proof}
    The generator corresponding to $F_j$ changes the sign of one ambient coordinate.
    This map has determinant $-1$ on the ambient space, while it preserves the ordered normal frame up to determinant $+1$.
    Therefore this generator reverses the orientation of $\RZ_{P_m}$. 
    A product of such generators preserves the orientation exactly when the number of chosen generators is even.
    This is precisely the condition $\langle \one_m,\tau\rangle=0$.
\end{proof}

Let $M$ be a small cover over $P_m$.
Since $M$ is obtained from $\RZ_{P_m}$ by a $\Z_2^{m-2}$ quotient, we have $\chi(\RZ_{P_m}) = 2^{m-2} \chi(M)$.
If $M$ is orientable, then the genus of $M$ is $\frac{m-2}{2}$.
In particular, there is no orientable small cover over $P_m$ for odd $m$.

Now, let us consider the case when $P=P_{m_1}\times\cdots\times P_{m_n}$.
Thus $P$ is a simple polytope of dimension $2n$.
For each $i=1,\ldots,n$, let $E_{i,1},\ldots,E_{i,m_i}$ be the facets of the polygon $P_{m_i}$, ordered cyclically.
We use the convention $E_{i,m_i+1}=E_{i,1}$.
The facets of $P$ are
\begin{equation}\label{eq:facets_index}
    F_{i,j} = P_{m_1}\times\cdots\times P_{m_{i-1}} \times E_{i,j} \times P_{m_{i+1}}\times\cdots\times P_{m_n}.
\end{equation}

Hence the number of facets of $P$ is $m=m_1+\cdots+m_n$.

The real moment-angle manifold over $P$ is 
\begin{equation}\label{eq:RZK_productpolygon}
    \RZ_P \cong \RZ_{P_{m_1}}\times\cdots\times \RZ_{P_{m_n}} = \Sigma_{\genus(m_1)} \times \cdots \times \Sigma_{\genus(m_n)}.
\end{equation}

Let $\lambda\colon \Z_2^m \to \Z_2^{2n}$ be a linear map.
After ordering the facets factor by factor, the matrix of $\lambda$ has the
block form
$$
    \lambda= (\lambda_{i,j}) = \begin{pmatrix} \Lambda_1 & \Lambda_2 & \cdots & \Lambda_n \end{pmatrix},
$$
where
$$
    \Lambda_i= \begin{pmatrix} \lambda_{i,1} & \lambda_{i,2} & \cdots & \lambda_{i,m_i}\end{pmatrix} \in \cM_{2n\times m_i}(\Z_2).
$$
The nonsingularity condition for $\lambda$ can be written only in terms of vertices of~$P$.
For a tuple $\bj=(j_1,\ldots,j_n)$ with $1\leq j_i\leq m_i$, the corresponding vertex is the intersection of the $2n$ facets $F_{1,j_1},F_{1,j_1+1},    \ldots, F_{n,j_n},F_{n,j_n+1}$.
Define the vertex matrix
$$
    B_{\mathbf j}(\lambda) =
        \begin{pmatrix}
            \lambda_{1,j_1} & \lambda_{1,j_1+1} & \lambda_{2,j_2} & \lambda_{2,j_2+1} & \cdots & \lambda_{n,j_n} & \lambda_{n,j_n+1}
        \end{pmatrix}.
$$
This is a $2n\times 2n$ matrix over $\Z_2$.
From the definition of the nonsingularity condition for a characteristic map, we have the following lemma.
\begin{lemma}\label{lem:nonsingularity-product-polygons}
    The map $\lambda$ is a mod $2$ characteristic map over
    $P=P_{m_1}\times\cdots\times P_{m_n}$ if and only if
    $$
        \det B_{\mathbf j}(\lambda)=1
    $$
    for every tuple $\bj =(j_1,\ldots,j_n)$ with $1\leq j_i\leq m_i$.
\end{lemma}

\subsection{Symplectic and Hodge-theoretic terminology}

Let $N$ be a closed smooth manifold of real dimension $2n$.
A symplectic structure on $N$ is a closed nondegenerate $2$-form.
If $\omega$ is such a form, then
$$
    [\omega]^n\neq 0 \in H^{2n}(N;\R).
$$
Thus a symplectic manifold has a nonzero top power of a degree-two cohomology class.

We say that $N$ is \emph{cohomologically symplectic}, or \emph{c-symplectic}, if there exists $a\in H^2(N;\Q)$ such that
$$
    a^n\neq 0 \in H^{2n}(N;\Q).
$$
Since the condition $a^n \neq 0$ is open in $H^2(N;\R)$ and rational classes are dense, the existence of a symplectic form implies the existence of such a rational class.
Every symplectic manifold is c-symplectic.

A natural next step beyond symplectic geometry is to ask whether a given manifold admits an integrable complex structure, or even a K\"ahler or projective structure.
These stronger structures often impose additional topological and cohomological restrictions.
A compact complex manifold is called \emph{K\"ahler} if it admits a K\"ahler metric.
A compact complex manifold is called \emph{projective} if it is biholomorphic to a smooth complex projective variety.
It is standard that
$$
    \text{projective} \Longrightarrow \text{K\"ahler} \Longrightarrow \text{symplectic} \Longrightarrow \text{c-symplectic};
$$
see, for example, \cite{Huybrechts2005book}.

Throughout the paper, by a \emph{smooth projective model} of a smooth manifold~$M$, we mean a smooth projective variety~$X$ whose underlying smooth manifold is diffeomorphic to~$M$.

If $X$ is a smooth projective variety of complex dimension~$n$, then its complexified cotangent bundle splits as
$$
    T_X^\ast\otimes \C = T_X^{\ast\,1,0}\oplus T_X^{\ast\,0,1}.
$$
Accordingly, for each $r$ one has a decomposition of complex differential forms
$$
    \Omega_X^r\otimes \C=\bigoplus_{p+q=r}\Omega_X^{p,q},
$$
where $\Omega_X^{p,q}$ denotes the space of smooth differential forms of type $(p,q)$.
If $X$ is K\"ahler, and in particular if $X$ is smooth projective, then Hodge theory shows that this decomposition descends to cohomology and gives a natural decomposition
$$
    H^r(X;\C)=\bigoplus_{p+q=r} H^{p,q}(X),
$$
called the \emph{Hodge decomposition}; see, for example, \cite{Voisin2007book}.
Thus the cohomology of~$X$ is decomposed into pieces of type $(p,q)$, reflecting the complex structure of~$X$.

We write
$$
    h^{p,q}(X):=\dim_\C H^{p,q}(X)
$$
for the \emph{Hodge numbers} of $X$.
The numbers $h^{p,q}(X)$ are usually arranged in a diamond-shaped array, called the \emph{Hodge diamond} of $X$.
They refine the rational Betti numbers, since
$$
    b_i(X)=\sum_{p+q=i} h^{p,q}(X).
$$
In general, the Hodge diamond contains more information than the Betti numbers and need not be determined by them.

\section{The factor-compatible class and projective models}\label{sec-products}

From this section on, we study small covers of even dimension.
Accordingly, throughout this section we consider a simple polytope of the form
$$
    P=\prod_{i=1}^n P_{m_i},
$$
where $P_{m_i}$ is an $m_i$-gon.
Thus the corresponding small covers have real dimension~$2n$.
We write $m=\sum_{i=1}^n m_i$, and decompose the facet set of $P$ as
$$
    \fF(P)=\bigsqcup_{i=1}^n \{F_{i,1},\dots,F_{i,m_i}\},
$$
where the block $\{F_{i,1},\dots,F_{i,m_i}\}$ comes from the $i$th polygon factor as in \eqref{eq:facets_index}.

We identify each facet $F_{i,j}$ with the corresponding standard basis vector of $\Z_2^m$.
With this convention, for each $i=1, \ldots, n$, we set the factor weights
$$
    \chi_i = F_{i,1}+ \dots + F_{i,m_i} \in \Z_2^m.
$$
We also write
$$
    \one_m= \sum_{F \in \fF(P)} F = \sum_{i=1}^n \chi_i \in \Z_2^m.
$$

Let $M=M(P,\lambda)$ be a small cover over $P$, and put $G=\ker\lambda$.
By Corollary~\ref{cor:NN}, $M$ is orientable if and only if $\one_m \in \row(\lambda)$.

Let $i_1, \ldots, i_r$ be the indices of the square factors.
For $a=1,\ldots, r$, define two opposite-pair weights
$$
    \delta_a^+=F_{{i_a},1}+F_{{i_a},3} \quad \text{ and } \quad \delta_a^-=F_{{i_a},2}+F_{{i_a},4}.
$$
If the $2r$ opposite-pair weights $\delta_1^+,\delta_1^-, \ldots, \delta_r^+,\delta_r^-$ can be partitioned into $r$ pairs $\{\omega_1,\omega_1'\},\dots,\{\omega_r,\omega_r'\}$ such that
$$
    \omega_\ell+\omega_\ell'\in \row(\lambda) \qquad  \text{for all } \ell=1,\dots,r,
$$
then we permute the interval directions and regroup them in such a way that $\omega_\ell$ and $\omega_\ell'$ form the two interval directions of the $\ell$th new square factor.

\begin{definition}\label{def:factor-compatible}
Let $M=M(P,\lambda)$ be a small cover over $P=\prod_{i=1}^n P_{m_i}$. 
We call $M$ \emph{factor-compatible} if the following condition holds.

After a permutation of the interval directions in the cube part
$$
    \prod_{m_i=4} P_4 \cong (\Delta^1)^{2r}
$$
and after regrouping them into $r$ square factors, the total facet weight $\chi_i$ of every resulting polygon factor satisfies
$$
    \chi_i\in\row(\lambda).
$$
After this regrouping, we relabel the resulting polygon factors and denote their total facet weights again by $\chi_i$.
Whenever a factor-compatible small cover is considered, such a regrouping is fixed once and for all.
\end{definition}

This regrouping is induced by a combinatorial automorphism of the cube part, namely by a permutation of its interval directions. 
It changes only the product decomposition of the cube part into square factors, and not the underlying real moment-angle manifold or the resulting quotient smooth manifold.

By Lemma~\ref{lem:orientation-character-polygon}, the condition says that the quotient action preserves the orientations of the surface factors.
The following lemma shows that every polygon factor in a factor-compatible small cover must have an even number of sides.

\begin{lemma}\label{lem:odd-factor-not-compatible}
Let $P = P_{m_1} \times \cdots \times P_{m_n}$ be a product of polygons, and let $\lambda$ be a mod~$2$ characteristic map over $P$.
If $\chi_i \in \row(\lambda)$, then $m_i$ is even.
\end{lemma}
\begin{proof}
Since $\chi_i \in \row(\lambda)$, there is a linear functional
    $$
        \varphi\colon \Z_2^{2n}\to \Z_2
    $$
whose values on the columns of $\lambda$ are prescribed by $\chi_i$.
    
Choose a vertex in the product of all factors except the $i$th one.
Equivalently, choose one adjacent pair of facets from each of those other polygon factors.
Let $W \subset \Z_2^{2n}$ be the span of the corresponding $2n-2$ characteristic vectors.
By the characteristic condition, these vectors are linearly independent, so $\dim W=2n-2$.
Moreover, $W\subset \ker\varphi$.
    
Now pass to the quotient $\Z_2^{2n}/W$.
The images of the characteristic vectors of the facets of the $i$th polygon all have $\varphi$-value $1$.
Hence they lie in the affine line
    $$
        \{v\in \Z_2^{2n}/W\mid \varphi(v)=1\},
    $$
which has two elements.
For two adjacent facets of the $i$th polygon, the characteristic condition implies that their images in this quotient are distinct.
Thus the two possible images must alternate as one goes around the polygon.
This is possible only when $m_i$ is even.
\end{proof}

\begin{remark}\label{rem:pure-cube-case}
In the pure cube case, the above definition recovers the usual symplectic condition for real Bott manifolds. 
Indeed, writing $(P_4)^n\cong(\Delta^1)^{2n}$, the condition says that the $2n$ interval directions can be paired so that the two weights in each pair define the same character on $G=\ker\lambda$. 
This is equivalent to the existence of a $G$-invariant translation-invariant symplectic form on the torus cover $\RZ_P\cong T^{2n}$. 
By Ishida's theorem, this agrees with the symplectic, equivalently K\"ahler, real Bott condition.
\end{remark}

We now prove that the factor-compatible condition gives projective models.

\begin{theorem}\label{thm:factor-compatible-projective}
Let $P=P_{m_1}\times\cdots\times P_{m_n}$ be a product of polygons, and let $M=M(P,\lambda)$ be a factor-compatible small cover.
Then $M$ admits a smooth projective model.
More precisely, there exists a smooth projective variety $X$ whose underlying smooth manifold is diffeomorphic to $M$.

In particular, the underlying smooth manifold of every factor-compatible small cover admits projective, K\"ahler, symplectic, and c-symplectic structures.
\end{theorem}

\begin{proof}
After the fixed regrouping in Definition~\ref{def:factor-compatible}, we write the resulting product of polygons again as $P=P_{m_1}\times\cdots\times P_{m_n}$.
Then
$$
    \RZ_P\cong \Sigma_{\genus(m_1)}\times\cdots\times \Sigma_{\genus(m_n)}.
$$
The group $G=\ker\lambda$ acts freely on this product, and $M\cong \RZ_P/G$.
    
The condition $\chi_i\in \row(\lambda)$ means that every element of $G$ has even parity on the facets coming from the $i$th polygon factor.
By Lemma~\ref{lem:orientation-character-polygon}, this is equivalent to saying that every element of $G$ acts orientation-preservingly on the surface factor $\Sigma_{\genus(m_i)}\cong \RZ_{P_{m_i}}$.
    
Now choose any Riemannian metric on $\Sigma_{\genus(m_i)}$.
Since the image of $G$ in the diffeomorphism group of this factor is finite, we may average the metric over this finite group.
The resulting metric is $G$-invariant.
Together with the fixed orientation of $\Sigma_{\genus(m_i)}$, this metric determines a complex structure on the surface.
Because the elements of $G$ preserve both the metric and the orientation, they are holomorphic with respect to this complex structure.
Thus $\Sigma_{\genus(m_i)}$ becomes a compact Riemann surface.
Equivalently, it is a smooth projective curve, which we denote by $\cC_i$, whose genus is $\genus(m_i)$.

Therefore, we obtain a smooth projective variety $Y=\cC_1\times\cdots\times \cC_n$ with a holomorphic action of $G$.
Under the product decomposition $\RZ_P\cong \RZ_{P_{m_1}}\times\cdots\times \RZ_{P_{m_n}}$, the canonical $\Z_2^m$-action is factorwise. 
The complex structures chosen above change only the geometric structures on the same smooth factors, not the underlying smooth $G$-action. 
Hence the resulting holomorphic $G$-action on $Y$ is free because the original smooth $G$-action on $\RZ_P$ is free.
Therefore
$$
    X:=Y/G
$$
is a smooth compact complex manifold.
Since $Y$ is projective and $G$ acts freely and holomorphically on $Y$, the quotient $X = Y/G$ is a smooth projective variety whose underlying smooth manifold is diffeomorphic to $M$.
Indeed, if $L$ is an ample line bundle on $Y$, then the tensor product of its $G$-translates admits a natural $G$-linearization and descends to an ample line bundle on~$X$.
\end{proof}

\section{Some obstructions to symplecticity}\label{sec-symplectic-obstructions}

Theorem~\ref{thm:factor-compatible-projective} shows that every factor-compatible small cover over a product of polygons is symplectic.
The converse is known for products of two polygons by~\cite{Choi2026}.
In dimensions at least six, the converse is not known in general.
In this section we record two elementary obstructions.

\begin{proposition}\label{prop:triangle-factor-not-csymp}
Let $P=P_{m_1}\times\cdots\times P_{m_n}$ be a product of polygons.
If one of the factors is a triangle, then no small cover over $P$ is c-symplectic.
In particular, no small cover over $P$ is symplectic.
\end{proposition}

\begin{proof}
Suppose that the $i$th factor is a triangle, so $m_i=3$.
Let $M = \RZ_P/G$ be a small cover over $P$ with $G=\ker \lambda$.
By \eqref{eq:RZK_productpolygon}, we have
$$
    \RZ_P \cong \Sigma_{\genus(m_1)}\times\cdots\times\Sigma_{\genus(m_n)}.
$$
The $i$th factor is $\Sigma_{\genus(3)}=S^2$.

By Lemma~\ref{lem:odd-factor-not-compatible}, we have $\chi_i\notin \row(\lambda)$.
Hence Lemma~\ref{lem:row-kernel-duality} gives an element $\tau\in G$ such that $\langle \chi_i,\tau\rangle=1$.
By Lemma~\ref{lem:orientation-character-polygon}, the element $\tau$ reverses the orientation of the $S^2$~factor.

Assume, for a contradiction, that $M$ is c-symplectic.
Then there exists $a\in H^2(M;\Q)$ such that $a^n\neq 0$.
Let $\pi\colon \RZ_P\to M$ be the finite covering map, and put $\widetilde a=\pi^\ast a$.
The transfer map shows that $\pi^\ast$ is injective over $\Q$.
Thus $\widetilde a^n\neq 0$.
Moreover, $\widetilde a$ is $G$-invariant.

Write $\RZ_P=S^2\times Y'$, where $Y'=\prod_{j\neq i}\Sigma_{\genus(m_j)}$.
Let $u\in H^2(S^2;\Q)$ be the orientation class.
Since $H^1(S^2;\Q)=0$, the K\"unneth decomposition gives
$$
    H^2(S^2\times Y';\Q) = \Q u\oplus H^2(Y';\Q).
$$
The element $\tau$ acts by $-1$ on $\Q u$.
Hence a $G$-invariant degree two class has zero $u$-component.
Therefore $\widetilde a\in H^2(Y';\Q)$.
But $Y'$ has real dimension $2n-2$.
Thus $\widetilde a^n=0$, a contradiction.
Hence $M$ is not c-symplectic.
\end{proof}

\begin{proposition}\label{prop:all-odd-polygons-not-orientable}
Let
$$
    P=P_{m_1}\times\cdots\times P_{m_n}
$$
be a product of polygons.
If all $m_i$ are odd, then no small cover over $P$ is orientable.
Consequently, no small cover over $P$ is symplectic.
\end{proposition}

\begin{proof}
    Let $M=M(P,\lambda)$ be a small cover over $P$.
    Write $\lambda_{i,j}=\lambda(F_{i,j})\in \Z_2^{2n}$, where the index $j$ is taken cyclically modulo $m_i$.
    
    For each tuple $\bj=(j_1,\ldots,j_n)$ with $1\leq j_i\leq m_i$, put
    $$
        B_{\bj}(\lambda) = \begin{pmatrix} \lambda_{1,j_1} & \lambda_{1,j_1+1} & \cdots & \lambda_{n,j_n} & \lambda_{n,j_n+1} \end{pmatrix}.
    $$
    By Lemma~\ref{lem:nonsingularity-product-polygons}, we have $\det B_{\bj}(\lambda)=1 \in \Z_2$ for every $\bj$.
    Since all $m_i$ are odd, it follows that
    \begin{equation}\label{eq:odd-det-sum}
        \sum_{\bj}\det B_{\bj}(\lambda) = m_1\cdots m_n = 1 \in \Z_2.
    \end{equation}
    
    Suppose, for a contradiction, that $M$ is orientable.
    By Corollary~\ref{cor:NN}, there exists a linear functional $\varphi\colon \Z_2^{2n}\to \Z_2$ such that $\varphi(\lambda_{i,j})=1$ for every $i$ and $j$.
    Choose $x\in \Z_2^{2n}$ with $\varphi(x)=1$.
    Then each characteristic vector can be written as $\lambda_{i,j}=x+e_{i,j}$ with $e_{i,j}\in \ker\varphi$.
    For each $i$, define 
    $$
        T_i=\sum_{j=1}^{m_i} \lambda_{i,j}\wedge \lambda_{i,j+1} \in \Lambda^2 (\Z_2^{2n}).
    $$
    Here the terms involving $x$ cancel because the indices are cyclic.
    Since $x \wedge x = 0$ and $e_{i,j} \wedge x = x \wedge e_{i,j}$ over $\Z_2$, we get
    \begin{align*}
        T_i &=\sum_{j=1}^{m_i}\bigl(x\wedge e_{i,j+1} + x\wedge e_{i,j} + e_{i,j}\wedge e_{i,j+1}\bigr)  \\
            &=\sum_{j=1}^{m_i}e_{i,j}\wedge e_{i,j+1} \in \Lambda^2(\ker\varphi).
    \end{align*}
    Therefore $T_1\wedge\cdots\wedge T_n \in \Lambda^{2n}(\ker\varphi)$.
    Since $\dim \ker \varphi=2n-1$, we have 
    \begin{equation}\label{eq:wedge-zero}
        T_1\wedge\cdots\wedge T_n=0.
    \end{equation}
    Let $\Omega\in \Lambda^{2n}(\Z_2^{2n})^\ast$ be the determinant form with respect to the standard basis of~$\Z_2^{2n}$.
    By multilinearity,
    \begin{align*}
        \Omega(T_1\wedge\cdots\wedge T_n) &= \sum_{\bj} \Omega(\lambda_{1,j_1}\wedge \lambda_{1,j_1+1} \wedge\cdots\wedge \lambda_{n,j_n}\wedge \lambda_{n,j_n+1}    ) \\
            &= \sum_{\bj}\det B_{\bj}(\lambda).
    \end{align*}
    By \eqref{eq:odd-det-sum}, the right-hand side is $1$.
    On the other hand, \eqref{eq:wedge-zero} implies that the left-hand side is $0$.
    This is a contradiction.
    
    Therefore $M$ is not orientable.
\end{proof}

\section{Hodge-theoretic rigidity of the projective models}\label{sec-hodge-mod2}

In Section~\ref{sec-products}, we constructed a smooth projective model for every factor-compatible small cover over a product of polygons.
This allows us to study such small covers not only as smooth manifolds, but also as projective varieties.
The main point of this section is that the factor-compatible projective models form a rigid class from this point of view.
In the sequel, whenever a factor-compatible small cover is considered, we replace the cube part by the fixed regrouped product decomposition from Definition~\ref{def:factor-compatible}.

Let $M=M(P,\lambda)$ be a factor-compatible small cover over $P=P_{m_1}\times\cdots\times P_{m_n}$, and let $X=Y/G$ and $Y=\cC_1\times\cdots\times \cC_n$ be the smooth projective model from Theorem~\ref{thm:factor-compatible-projective}.
Here $G=\ker\lambda$ is a finite elementary abelian $2$-group.
For each factor, set
$$
    V_i:=H^{1,0}(\cC_i).
$$
The corresponding curve has $\dim_\C V_i=\genus(m_i)$.

Since $G$ is elementary abelian, every complex character of $G$ is real-valued.
Therefore $V_i\cong \overline{V_i}$ as $G$-modules.

For each subset $S\subset [n]$ and each $0\le s\le n$, set
$$
    t_S:=\dim_\C\left(\bigotimes_{i\in S} V_i\right)^G \quad \text{ and } \quad T_s:=\sum_{|S|=s} t_S.
$$
We write
    $$
        P_X(t)=\sum_k b_k(X)t^k
    $$
for the rational Poincar\'e polynomial of $X$.

\begin{theorem}\label{thm:hodge-polynomial-factor-compatible}
The Hodge polynomial of $X$ is
    $$
    H_X(u,v)
    =
    \sum_{p,q\ge 0} h^{p,q}(X)u^p v^q
    =
    \sum_{S\subset [n]} t_S\,(u+v)^{|S|}(1+uv)^{n-|S|}.
    $$
Equivalently,
    $$
    H_X(u,v)
    =
    \sum_{s=0}^n T_s\,(u+v)^s(1+uv)^{n-s}.
    $$
The Poincar\'e polynomial is
    $$
    P_X(t)
    =
    \sum_{s=0}^n T_s\,(2t)^s(1+t^2)^{n-s}.
    $$
\end{theorem}

\begin{proof}
    Let $\pi:Y\to X=Y/G$ be the quotient map.
    Since the action of $G$ on $Y$ is free and holomorphic, $\pi$ is a finite \etale cover.
    Hence, by the transfer argument and the compatibility of the Hodge decomposition with the $G$-action,
    $$
        H^{p,q}(X)\cong H^{p,q}(Y)^G.
    $$
    For each curve factor,
    $$
        H^\ast(\cC_i;\C) = \C\oplus V_i\oplus \overline{V_i}\oplus \C(-1),
    $$
    where the summands have bidegrees $(0,0)$, $(1,0)$, $(0,1)$, and $(1,1)$, respectively.
    The one-dimensional summands $\C$ and $\C(-1)$ are trivial $G$-representations.
    Indeed, $H^0(\cC_i;\C)$ is always trivial, and the $G$-action is orientation-preserving on the complex curve~$\cC_i$, so it acts trivially on $H^2(\cC_i;\C)\cong H^{1,1}(\cC_i)$.
    Thus
    $$
        H^\ast(Y;\C) \cong \bigotimes_{i=1}^n \left( \C\oplus V_i\oplus \overline{V_i}\oplus \C(-1) \right).
    $$
    
    Since $V_i$ and $\overline{V_i}$ are isomorphic as $G$-modules, the dimension of the invariant part is $t_S$ for each choice of whether the degree-one class from a factor in $S\subset [n]$ has type $(1,0)$ or $(0,1)$.
    The factors outside $S$ contribute either degree zero or degree two classes, giving the factor
    $$
        (1+uv)^{n-|S|}.
    $$
    The factors in $S$ contribute the factor
    $$
        (u+v)^{|S|}.
    $$
    This gives
    $$
        H_X(u,v) = \sum_{S\subset [n]} t_S (u+v)^{|S|}(1+uv)^{n-|S|}.
    $$
    Grouping the summands by $|S|$ gives the formula involving $T_s$.
    Finally, setting $u=v=t$ gives the Poincar\'e polynomial.
\end{proof}

For $u\in \Z_2^m$ and $\tau \in G = \ker\lambda$, put
$$
    \varepsilon_u (\tau) = \langle u, \tau \rangle.
$$
Then $\varepsilon_u$ is an element of $\widehat G:=\Hom(G,\Z_2)$.
By Lemma~\ref{lem:row-kernel-duality} applied to $u+u'$, two vectors $u,u'\in \Z_2^m$ define the same character on $G$ if and only if $u+u'\in \row(\lambda)$.

For each polygon factor, put $I_i=\{F_{i,1},\dots,F_{i,m_i}\}$ and regard a subset of $I_i$ as the corresponding vector in $\Z_2^m$.
Set
$$
    L_i:=\row(\lambda)\cap \Z_2^{I_i}.
$$
Then, for $u,u'\subset I_i$, one has
$$
    \varepsilon_u=\varepsilon_{u'} \iff u+u'\in L_i.
$$
In particular, $\varepsilon_u=0$ if and only if $u\in L_i$.

For $\rho\in\widehat G$, let $\C_\rho$ denote the one-dimensional complex representation on which $\tau \in G$ acts by multiplication by~$(-1)^{\rho(\tau)}$.
Write
\begin{equation}\label{eq:V_i}
    V_i\cong \bigoplus_{\rho\in\widehat G} a_{i,\rho} \C_\rho .
\end{equation}

The multiplicities~$a_{i,\rho}$ are computed as follows.
For a subset $u\subset I_i$, set
$$
    d_i(u):=\dim_\Q \widetilde H^0\bigl((P_{m_i})_u;\Q\bigr).
$$
By \eqref{eq:rational_additive_formula_RZP}, the rational first cohomology decomposes as a $G$-representation as
$$
    H^1(\RZ_{P_{m_i}};\Q) \cong \bigoplus_{u\subset I_i} d_i(u)\,\Q_{\varepsilon_u}. 
$$
After complexification, and after choosing the $G$-invariant complex structure on this surface, we have
$$
    H^1(\cC_i;\C) = H^{1,0}(\cC_i)\oplus H^{0,1}(\cC_i) =V_i\oplus \overline{V_i}.
$$
The $G$-action is holomorphic, so $\overline{V_i}$ is the complex conjugate representation of $V_i$.
Since $G$ is an elementary abelian $2$-group, all its complex characters are real-valued.
Therefore $V_i$ and $\overline{V_i}$ have the same character multiplicities.

It follows that the multiplicity of a character $\rho$ in $H^1(\cC_i;\C)$ is twice its multiplicity in $V_i$.
Hence
\begin{equation}\label{eq:a_irho}
    a_{i,\rho} = \frac12 \sum_{\substack{u\subset I_i\\ \varepsilon_u=\rho}} d_i(u).
\end{equation}
In particular, $a_{i,0}=\frac12 \sum_{u\in L_i} d_i(u)$.
The integrality of the right-hand side follows from the preceding decomposition.

\begin{proposition}\label{prop:tS-from-lambda}
    For every subset $S\subset[n]$, one has
    $$
        t_S= \sum_{\substack{(\rho_i)_{i\in S}\in \widehat G^S\\ \sum_{i\in S}\rho_i=0}} \prod_{i\in S} a_{i,\rho_i} = \frac{1}{|G|} \sum_{g\in G} \prod_{i\in S} \left( \sum_{\rho\in \widehat G} a_{i,\rho}(-1)^{\rho(g)} \right).
    $$
\end{proposition}
\begin{proof}
    By definition, $t_S=\dim_\C\left(\bigotimes_{i\in S}V_i\right)^G$.
    Using the decomposition~\eqref{eq:V_i}, we have
    $$
        \bigotimes_{i\in S}\C_{\rho_i}\cong\C_{\sum_{i\in S}\rho_i}.
    $$
    Thus such a tensor factor contributes to the invariant part exactly when $\sum_{i\in S}\rho_i=0$.
    This gives the first formula.
The second formula is the standard averaging formula for the dimension of invariant vectors under a finite group action.
\end{proof}

\begin{example}\label{ex:hodge-diamond-nontrivial-character}
    Let $P=P_6\times P_6$.
    Consider the characteristic matrix
    $$
        \lambda=
        \begin{pNiceArray}{cccccc|cccccc}[first-row,first-col]
            &
            \scriptstyle F_{1,1}&\scriptstyle F_{1,2}&\scriptstyle F_{1,3}
            &\scriptstyle F_{1,4}&\scriptstyle F_{1,5}&\scriptstyle F_{1,6}
            &
            \scriptstyle F_{2,1}&\scriptstyle F_{2,2}&\scriptstyle F_{2,3}
            &\scriptstyle F_{2,4}&\scriptstyle F_{2,5}&\scriptstyle F_{2,6}
            \\
            \lambda_1&
            1&0&1&0&1&0&
            1&0&1&0&0&0
            \\
            \lambda_2&
            1&1&1&1&1&1&
            0&0&0&0&0&0
            \\
            \lambda_3&
            0&0&0&0&0&0&
            1&0&1&0&1&0
            \\
            \lambda_4&
            0&0&0&0&0&0&
            1&1&1&1&1&1
        \end{pNiceArray}.
    $$
    One checks directly that this is a characteristic matrix, and the small cover is factor-compatible.
    
    Let $X=(\cC_1\times \cC_2)/G$ be the projective model.
    Both curves have genus $\genus(6)=17$.
    The rows of $\lambda$ give the following relations in $\widehat G=\Hom(G,\Z_2)\cong \Z_2^{12}/\row(\lambda)$:
    $$
        F_{1,1}+F_{1,3}+F_{1,5}+F_{2,1}+F_{2,3}=0,
    $$
    $$
        F_{1,1}+\cdots+F_{1,6}=0,
    $$
    $$
        F_{2,1}+F_{2,3}+F_{2,5}=0, \qquad F_{2,1}+\cdots+F_{2,6}=0.
    $$
    In particular, $F_{1,1}+F_{1,3}+F_{1,5} = F_{2,1}+F_{2,3} = F_{2,5}$ as characters of $G$.
    
    We now compute $t_S$.
    For the first factor, $L_1=\row(\lambda)\cap \Z_2^{I_1}=\{0,\chi_1\}$.
    Since $d_1(0)=d_1(\chi_1)=0$, we have $a_{1,0}=0$ and $t_{\{1\}}=0$.
    
    Let $\rho:=F_{1,1}+F_{1,3}+F_{1,5}\in\widehat G$.
    The two subsets of $I_1$ representing $\rho$ are
    $$
        F_{1,1}+F_{1,3}+F_{1,5} \quad\text{ and }\quad F_{1,2}+F_{1,4}+F_{1,6}.
    $$
    Each gives three disjoint sides, hence three connected components, in the hexagon, so both have $d_1=2$.
    Thus, by~\eqref{eq:a_irho},
    $$
        a_{1,\rho}=\frac12(2+2)=2.
    $$
    
    For the second factor,
    $$
        L_2=\row(\lambda)\cap \Z_2^{I_2} = \{0, F_{2,1}+F_{2,3}+F_{2,5}, F_{2,2}+F_{2,4}+F_{2,6},\chi_2\}.
    $$
    Since the corresponding $d_2$-values are $0$, $2$, $2$, and $0$, we have
    $$
        a_{2,0}=\frac12(0+2+2+0)=2, \quad \text{ and } \quad t_{\{2\}}=2.
    $$
    
    The same character $\rho$ also appears in the second factor, since $\rho=F_{2,1}+F_{2,3}=F_{2,5}$ in $\widehat G$.
    The subsets of $I_2$ representing $\rho$ are the elements of $(F_{2,1}+F_{2,3})+L_2$.
    Their $d_2$-values are two $1$'s and two $0$'s.
    Thus,
    $$
        a_{2,\rho}=\frac12(1+0+1+0)=1.
    $$
    Here $d_i(u)$ is one less than the number of connected components of $(P_6)_u$ when $u\ne\varnothing$, and $d_i(\varnothing)=0$.
    For the remaining character classes, one groups the subsets of $I_i$ into cosets modulo $L_i$ and applies formula~\eqref{eq:a_irho} to each coset.
    This coset check shows that the only nonzero character appearing in both $V_1$ and $V_2$ is the character~$\rho$ above.
    Hence no other common character contributes to $t_{\{1,2\}}$.
    Therefore
    $$
        t_{\{1,2\}} = a_{1,\rho}a_{2,\rho} = 2.
    $$
    Consequently,
    $$
        T_0=1,\qquad T_1=2,\qquad T_2=2.
    $$
    
    By Theorem~\ref{thm:hodge-polynomial-factor-compatible},
    \begin{align*}
        H_X(u,v) &= (1+uv)^2 + 2(u+v)(1+uv) + 2(u+v)^2 \\
            &= 1+2(u+v)+2(u^2+v^2)+6uv + 2(u^2v+uv^2)+u^2v^2.  
    \end{align*}
    Therefore the Hodge diamond is
    $$
        \begin{array}{ccccc}
        &&1&&\\
        &2&&2&\\
        2&&6&&2\\
        &2&&2&\\
        &&1&&
        \end{array}.
    $$
\end{example}

\begin{lemma}\label{lem:betti-determine-hodge}
The rational Betti numbers of $X$ determine its Hodge polynomial and hence its Hodge diamond.
\end{lemma}
\begin{proof}
By Theorem~\ref{thm:hodge-polynomial-factor-compatible},
    $$
        P_X(t)=\sum_{s=0}^n T_s\,(2t)^s(1+t^2)^{n-s}.
    $$
The polynomial $(2t)^s(1+t^2)^{n-s}$ has lowest nonzero degree $s$.
Therefore the coefficients $T_s$ can be recovered recursively from $P_X(t)$, starting with $s=0$ and increasing~$s$.
Thus the rational Betti numbers determine all $T_s$.
The formula in Theorem~\ref{thm:hodge-polynomial-factor-compatible} then determines the Hodge polynomial.
    \end{proof}

\begin{remark}\label{rem:hodge-independent-choices}
    In particular, the Hodge diamond of the projective model constructed above is independent of the auxiliary choices made in the construction.
    This includes the choices of invariant complex structures on the non-square surface factors and, in the square case, the choice of a compatible pairing whenever more than one such pairing is available.
\end{remark}

We now show that, for the projective models constructed above, the Hodge diamond is already determined by the mod~$2$ cohomology ring of the underlying small cover.

\begin{proposition}\label{prop:mod2-determines-rational-betti-products-polygons}
    Let $M=M(P,\lambda)$ be any small cover over a product of polygons.
    Then the graded mod~$2$ cohomology ring $H^\ast(M;\Z_2)$ determines the rational Betti numbers of $M$.
\end{proposition}

\begin{proof}
    Let $P=P_{m_1}\times\cdots\times P_{m_n}$.
    For a polygon, any union of sides is either the whole boundary circle or a disjoint union of intervals and points, hence has torsion-free integral cohomology. 
    For a product of polygons, the facet cover of~$P_\omega$ has contractible nonempty intersections, and its nerve is obtained from the corresponding one-factor nerves by joins. 
    Hence, by the nerve theorem and the Künneth formula for joins, $H^\ast(P_\omega; \Z)$ is torsion-free.
    
    By Corollary~\ref{thm:Cai-Choi-two-torsion}, the torsion subgroup of $H^\ast(M;\Z)$ is a direct sum of copies of~$\Z_2$.
    Hence the Bockstein spectral sequence associated with
    $$
        0\to \Z \xrightarrow{\times 2}\Z\to \Z_2\to 0
    $$
    collapses at the $E_2$-page, and this $E_2$-page is the mod~$2$ reduction of the free part of $H^\ast(M;\Z)$.
    
    The first differential in this Bockstein spectral sequence is $Sq^1$.
    Since $H^\ast(M;\Z_2)$ is generated in degree one as~\eqref{eq:mod2cohom_smallcover}, the operation $Sq^1$ is determined by the graded ring structure.
    Namely, for every degree-one class $x$, one has $Sq^1(x)=x^2$, and the Cartan formula extends this rule to the whole ring.
    Consequently the graded ring determines
    $$
        E_2^\ast=\ker Sq^1/\im Sq^1.
    $$
    Since there is no torsion of order greater than $2$, the dimensions of these $E_2$-terms are precisely the ranks of the free parts of $H^\ast(M;\Z)$, equivalently the rational Betti numbers of $M$.
    See~\cite[Lemma~8.1]{Choi-Masuda-Suh2010TAMS}.
\end{proof}

\begin{theorem}\label{thm:mod2-rigidity-factor-compatible}
    Let $M=M(P,\lambda)$ be a factor-compatible small cover over a product of polygons, and let $X$ be the projective model from Theorem~\ref{thm:factor-compatible-projective}.
    Then the graded mod~$2$ cohomology ring determines the Hodge diamond of $X$.
\end{theorem}
\begin{proof}
    By Proposition~\ref{prop:mod2-determines-rational-betti-products-polygons}, the graded mod~$2$ cohomology ring determines the rational Betti numbers of $M$.
    Since the underlying smooth manifold of $X$ is diffeomorphic to $M$, these are also the rational Betti numbers of $X$.
    By Lemma~\ref{lem:betti-determine-hodge}, the rational Betti numbers of $X$ determine the Hodge polynomial and hence the Hodge diamond of $X$.
\end{proof}

\begin{remark}
    Theorem~\ref{thm:mod2-rigidity-factor-compatible} is related to the mod~$2$ cohomological rigidity problem for small covers.
    For background on this problem, see Choi--Masuda--Suh~\cite{Choi-Masuda-Suh2011}.
    
    In general, the graded mod~$2$ cohomology ring is not a complete invariant of the diffeomorphism type~\cite{Masuda2010}.
    However, in the real Bott case, there are positive rigidity results~\cite{Kamishima-Masuda2009,Choi-Masuda-Oum2017}.
    Thus, although the graded mod~$2$ cohomology ring is not a complete invariant for small covers in general, it still carries substantial topological information in important subclasses.
    The author proved in \cite{Choi2026} that mod~$2$ cohomological rigidity holds for factor-compatible small covers of dimension~$4$.
    However, this is not known in higher dimensions.
    
    Theorem~\ref{thm:mod2-rigidity-factor-compatible} points in a related but different direction.
    Our result suggests that the graded mod~$2$ cohomology ring also carries meaningful geometric information.
    This suggests the broader question of how much geometric information about a small cover is encoded in its graded mod~$2$ cohomology ring.
\end{remark}

\section{A block triangular form for factor-compatible small covers}\label{sec-block-triangular}

Let $P=\prod_{i=1}^n P_{m_i}$ be a product of polygons, and let $M(P,\lambda)$ be a factor-compatible small cover over $P$.
Let $I_i = \{F_{i,1}, \ldots, F_{i,m_i}\}$ be the facet set from the $i$th polygon factor as in \eqref{eq:facets_index}.

\begin{definition}
    Let the rows be divided into $n$ blocks of size $2$, and let the columns be divided according to the polygon factors $I_1,\ldots,I_n$.
    A characteristic matrix is said to be \emph{block lower triangular} if it has the form
    $$
        \lambda'=
        \begin{pmatrix}
        \Lambda_{11} & 0            & \cdots & 0\\
        \Lambda_{21} & \Lambda_{22} & \cdots & 0\\
        \vdots       & \vdots       & \ddots & \vdots\\
        \Lambda_{n1} & \Lambda_{n2} & \cdots & \Lambda_{nn}
        \end{pmatrix},
    $$
    where $\Lambda_{ij}$ is a $2\times m_j$ matrix.
    We say that the diagonal blocks are \emph{polygonal characteristic blocks} if, for each $i$, $\Lambda_{ii}$ is a characteristic matrix on the polygon $P_{m_i}$; that is, for every adjacent pair $F_{i,k},F_{i,k+1}$, the corresponding two columns of $\Lambda_{ii}$ form a basis of $\Z_2^2$.
\end{definition}

We shall use the following elementary linear algebra lemmas.

\begin{lemma}[{\cite[Lemma~3.3]{Masuda-Panov2008}}] \label{lem:principal-minor lemma}
    Let $A=(a_{ij})\in M_n(\Z_2)$.
    Suppose that every principal minor of $A$, including the determinant of $A$ itself, is equal to $1$.
    Then, after a simultaneous permutation of rows and columns, $A$ is upper triangular with all diagonal entries equal to $1$.
\end{lemma}

\begin{lemma} \label{lem:colored triangularization}
    Let $E$ be an $n$-dimensional vector space over $\Z_2$, and let
    $$
        D_1,\ldots,D_n\subset E^\ast\setminus\{0\}
    $$
    be nonempty subsets.
    Suppose that every transversal~$d_1\in D_1,\ldots,d_n\in D_n$ is a basis of $E^\ast$.
    Then, after permuting the subsets~$D_i$, there exists a basis $\varepsilon_1,\ldots,\varepsilon_n$ of~$E^\ast$ such that
    $$
        D_i\subset \varepsilon_i+ \langle \varepsilon_{i+1},\ldots,\varepsilon_n\rangle \qquad \text{for every } i=1,\ldots,n.
    $$
    Here the span is interpreted as $0$ when $i=n$.
\end{lemma}
\begin{proof}
    We first prove that at least one of the sets $D_i$ is a singleton.
    Suppose not.
    Choose $a_i\in D_i$ and $b_i\in D_i\setminus\{a_i\}$ for each $i$.
    Since every transversal is a basis, $a_1,\ldots,a_n$ is a basis of $E^\ast$.
    Write $b_i=\sum_{j=1}^n c_{ji}a_j$, and let $C=(c_{ji})$.
        
    For any subset $S\subset\{1,\ldots,n\}$, replace $a_i$ by $b_i$ precisely for $i\in S$.
    The resulting $n$-tuple is again a transversal, hence a basis.
    The determinant of the corresponding change-of-basis matrix is therefore $1$.
    Expanding along the columns indexed by the complement of $S$, this determinant is the principal minor of $C$ indexed by $S$.
    Hence every principal minor of $C$ is $1$.
        
    By Lemma~\ref{lem:principal-minor lemma}, after a simultaneous permutation of rows and columns, $C$ is upper triangular with diagonal entries $1$.
    But then the first column of the permuted matrix is the first standard basis vector.
    This says that $b_i=a_i$ for the corresponding index~$i$, contradicting the choice $b_i\neq a_i$.
    Therefore some $D_i$ is a singleton.
        
    We now argue by induction on $n$.
    The case $n=1$ is immediate.
    For $n>1$, after permuting the $D_i$'s, assume $D_n=\{\varepsilon_n\}$.
    Let $\overline{E^\ast}=E^\ast/\langle \varepsilon_n\rangle$, and let $\overline D_i$ be the image of $D_i$ in $\overline{E^\ast}$, for $i=1,\ldots,n-1$.
    No element of any $\overline D_i$ is zero; otherwise some $D_i$ would contain $\varepsilon_n$, and choosing this element together with the element of~$D_n$ would give a dependent transversal.
        
    Moreover, every transversal from $\overline D_1,\ldots,\overline D_{n-1}$ is a basis of $\overline{E^\ast}$, because adjoining $\varepsilon_n$ gives a basis of $E^\ast$.
    By induction, after permuting $\overline D_1,\ldots,\overline D_{n-1}$, there is a basis $\overline\varepsilon_1,\ldots,\overline\varepsilon_{n-1}$ of~$\overline{E^\ast}$ such that
    $$
        \overline D_i\subset \overline\varepsilon_i+ \langle \overline\varepsilon_{i+1},\ldots, \overline\varepsilon_{n-1} \rangle.
    $$
    Choose arbitrary lifts $\varepsilon_1,\ldots,\varepsilon_{n-1}\in E^\ast$ of these basis elements.
    Then $\varepsilon_1,\ldots,\varepsilon_n$ is a basis of $E^\ast$, and the desired inclusions follow.
\end{proof}

\begin{theorem}\label{thm:blockize}
    Let $P=\prod_{i=1}^n P_{m_i}$ be a product of polygons, and let $\lambda$ be a factor-compatible characteristic matrix on $P$.
    Then, after permuting the polygon factors, $\lambda$ is Davis--Januszkiewicz equivalent to a block lower triangular characteristic matrix
    $$
        \lambda'=
        \begin{pmatrix}
        \Lambda_{11} & 0            & \cdots & 0\\
        \Lambda_{21} & \Lambda_{22} & \cdots & 0\\
        \vdots       & \vdots       & \ddots & \vdots\\
        \Lambda_{n1} & \Lambda_{n2} & \cdots & \Lambda_{nn}
        \end{pmatrix},
    $$
    where each $\Lambda_{ii}$ is a characteristic matrix on $P_{m_i}$.
    Moreover, the first row of each $\Lambda_{ii}$ may be chosen to be the all-one vector.
\end{theorem}
\begin{proof}
    By factor-compatibility, and by the convention fixed in Definition~\ref{def:factor-compatible}, we have $\chi_i\in\row(\lambda)$ for every factor $i$.
    Set $C=\langle \chi_1,\ldots,\chi_n\rangle\subset \row(\lambda)$.
    The supports of the $\chi_i$'s are disjoint, so $\dim C=n$.
    Since $\lambda$ is a characteristic matrix of a $2n$-dimensional small cover, $\dim \row(\lambda)=2n$.
    Hence $E:=\row(\lambda)/C$ has dimension~$n$.
        
    For each factor $i$ and each adjacent pair $F_{i,k},F_{i,k+1}$, define a linear functional $d_{i,k} \colon E\to \Z_2$ by
    $$
        d_{i,k}([w])=w(F_{i,k})+w(F_{i,k+1}).
    $$
    This is well-defined because adding $\chi_i$ changes both values $w(F_{i,k})$ and $w(F_{i,k+1})$ by $1$, while adding $\chi_j$ for $j\neq i$ does not affect these two coordinates.
    Put
    $$
            D_i=\{d_{i,1},\ldots,d_{i,m_i}\}\subset E^\ast.
    $$
        
    We claim that every transversal $d_{1,k_1}\in D_1,\ldots,d_{n,k_n}\in D_n$ is a basis of $E^\ast$.
    Choose the vertex of $P$ corresponding to the adjacent pairs
    $$
        \{F_{1,k_1},F_{1,k_1+1}\},\ldots,\{F_{n,k_n},F_{n,k_n+1}\}.
    $$
    The characteristic condition says that the restriction map $\row(\lambda)\to \Z_2^{2n}$ to these $2n$ facets is an isomorphism.
    Under this restriction, the subspace~$C$ maps to the subspace which is constant on each of the $n$ chosen adjacent pairs.
    Therefore, after quotienting by $C$, the induced map $E \to \Z_2^n$ is given by $[w]\mapsto \bigl(d_{1,k_1}([w]),\ldots,d_{n,k_n}([w])\bigr)$.
    This map is an isomorphism.
    This proves the claim.
    In particular, no element of any $D_i$ is zero.
        
    Apply Lemma~\ref{lem:colored triangularization} to $D_1,\ldots,D_n$.
    After permuting the polygon factors, there exists a basis $\varepsilon_1,\ldots,\varepsilon_n$ of~$E^\ast$ such that $D_i\subset\varepsilon_i+\langle \varepsilon_{i+1},\ldots,\varepsilon_n\rangle$ for all $i$.
    Let $\tau_1,\ldots,\tau_n$ be the dual basis of $E$, and choose arbitrary lifts $\widetilde\tau_i\in \row(\lambda)$ of $\tau_i$.
        
    Fix $i$.
    If $k>i$, then every element of $D_k$ lies in $\varepsilon_k+\langle\varepsilon_{k+1},\ldots,\varepsilon_n\rangle$.
    Hence every $d\in D_k$ satisfies $d(\tau_i)=0$.
    Thus, for every adjacent pair in the $k$-th polygon factor,
    $$
        \widetilde\tau_i(F_{k,j})+\widetilde\tau_i(F_{k,j+1})=0.
    $$
    Since the adjacency graph of a polygon is connected, the restriction of $\widetilde\tau_i$ to $I_k$ is constant.
    Hence it is either $0$ or $\chi_k|_{I_k}$.
    By adding suitable multiples of $\chi_k$, for~$k>i$, we may replace $\widetilde\tau_i$ by $\eta_i=\widetilde\tau_i+\sum_{k>i} c_{ik}\chi_k$ so that $\eta_i|_{I_k}=0$ for every $k>i$.
    Since the added terms lie in $C$, the class of $\eta_i$ in $E=\row(\lambda)/C$ is still $\tau_i$.
        
    Now consider the $i$th factor.
    Every element of $D_i$ has $\varepsilon_i$-coefficient equal to $1$.
    Hence, for every adjacent pair, $d_{i,j}(\tau_i)=1$.
    Therefore the two rows $\chi_i|_{I_i}$ and $\eta_i|_{I_i}$ define a characteristic matrix on the polygon $P_{m_i}$.
    Indeed, on every adjacent pair the two column vectors are $\binom{1}{a}$ and $\binom{1}{a+1}$, which are linearly independent over $\Z_2$.
        
    Finally, consider the following ordered basis of $\row(\lambda)$
    $$
            \chi_1,\eta_1,\chi_2,\eta_2,\ldots,\chi_n,\eta_n.
    $$
    It is a basis because $\chi_1,\ldots,\chi_n$ form a basis of~$C$, and $\eta_1,\ldots,\eta_n$ project to the basis $\tau_1,\ldots,\tau_n$ of $E=\row(\lambda)/C$.
        
    Let $\lambda'$ be the matrix whose rows are ordered as above.
    Since $\lambda'$ and $\lambda$ have the same row space, they differ by left multiplication by an element of $GL(2n,\Z_2)$.
    Hence they are Davis--Januszkiewicz equivalent.
        
    By construction, the row pair $(\chi_i,\eta_i)$ has zero restriction to $I_k$ whenever $k>i$.
    Therefore $\lambda'$ is block lower triangular.
    Moreover, the $i$th diagonal block is
    $$
        \Lambda_{ii}
        =
        \begin{pmatrix}
        \chi_i|_{I_i}\\
        \eta_i|_{I_i}
        \end{pmatrix},
    $$
    which is a characteristic matrix on $P_{m_i}$.
    The first row of $\Lambda_{ii}$ is $\one_{m_i}$.
    This proves the theorem.
\end{proof}

We now recall the block criterion for equivariant bundles in the real case.
This is the $\Z_2$-analogue of the corresponding criterion for quasitoric manifolds; see \cite[Lemma~3.2]{Choi-Park2016IJM}.

Here an equivariant bundle between small covers means a fiber bundle compatible with the locally standard $\Z_2$-torus actions, with respect to a surjective homomorphism of the acting groups.
In the product situation $B\times F$, we order the facets by
$$
    \fF(B\times F)=\fF(B)\sqcup \fF(F).
$$
A characteristic matrix on $B\times F$ is called a bundle block matrix over $B$ if, after splitting the rows according to $\dim B=b$ and $\dim F=f$, it has the form
$$
    \lambda=\begin{pmatrix} \Lambda_B & 0\\ * & \Lambda_F \end{pmatrix}.
$$
For such a matrix, the characteristic condition automatically implies that $\Lambda_B$ and $\Lambda_F$ are characteristic matrices on $B$ and $F$, respectively.

\begin{lemma}\label{lem:block-bundle-criterion}
    Let $\lambda$ be a bundle block matrix on $B\times F$.
    Then the projection $B\times F\to B$ induces an equivariant bundle $M(B\times F,\lambda) \to M(B,\Lambda_B)$ with fiber $M(F,\Lambda_F)$.
\end{lemma}

\begin{proof}
    Write elements of the acting group as pairs $(g_B,g_F)\in \Z_2^b\oplus\Z_2^f$.
    Define a map
    $$
        \Phi\colon M(B\times F,\lambda)\longrightarrow M(B,\Lambda_B)
    $$
    by
    $$
        \Phi([(g_B,g_F),b,f])=[g_B,b].
    $$
    This is well-defined.
    Indeed, if two points in $\Z_2^{b+f}\times B\times F$ are identified over a face $E_B\times E_F$ of $B\times F$, then their group difference lies in the subgroup generated by the columns of $\lambda$ corresponding to facets containing $E_B\times E_F$.
    After projection to $\Z_2^b$, this difference lies in the subgroup generated by the columns of $\Lambda_B$ corresponding to the facets of $B$ containing $E_B$.
    Hence the images are identified in $M(B,\Lambda_B)$.
    
    The map $\Phi$ is compatible with the projection $\Z_2^{b+f}\to \Z_2^b$ of acting groups, and therefore is equivariant in the usual weak sense.
    Over a local standard chart of the base small cover, the quotient relation splits into the base relation determined by $\Lambda_B$ and the fiber relation determined by $\Lambda_F$.
    Thus $\Phi$ is locally a product with fiber $M(F,\Lambda_F)$.
    Hence $\Phi$ is an equivariant bundle with fiber $M(F,\Lambda_F)$.
\end{proof}

The following corollary follows immediately from Theorem~\ref{thm:blockize} and Lemma~\ref{lem:block-bundle-criterion}.
\begin{corollary} \label{cor:iterated-equivariant-bundle}
    Every factor-compatible small cover over a product of polygons is, after the combinatorial square regrouping described above and, for the resulting product of polygons, up to Davis--Januszkiewicz equivalence, an iterated equivariant bundle whose fibers are orientable small covers over polygons.
\end{corollary}

More precisely, after the square regrouping and the D--J equivalence of Theorem~\ref{thm:blockize}, there is a tower
$$
    M_n\longrightarrow M_{n-1}\longrightarrow\cdots \longrightarrow M_1\longrightarrow \{\mathrm{pt}\},
$$
where $M_n$ represents the regrouped small cover, which has the same underlying smooth manifold as the original one, and the fiber of $M_i\to M_{i-1}$ is the small cover $M(P_{m_i},\Lambda_{ii})$ over the $i$th polygon.
Each fiber $M(P_{m_i},\Lambda_{ii})$ is a closed orientable surface of genus~$\frac{m_i-2}{2}$ as seen in Section~\ref{sec-preliminaries}.

\begin{remark}
    The tower in Corollary~\ref{cor:iterated-equivariant-bundle} should be viewed as a surface-type analogue of a generalized real Bott tower.
    Recall that generalized real Bott manifolds arise as small covers over products of simplices and admit iterated equivariant bundle structures whose fibers are real projective spaces.
    In particular, every small cover over a product of simplices is Davis--Januszkiewicz equivalent to a generalized real Bott manifold; see \cite{Choi-Masuda-Suh2010Quasitoric}.
\end{remark}


\begin{thebibliography}{10}

\bibitem{Cai-Choi2021}
Li~Cai and Suyoung Choi, \emph{Integral cohomology groups of real toric
  manifolds and small covers}, Mosc. Math. J. \textbf{21} (2021), no.~3,
  467--492. \MR{4277852}

\bibitem{Choi2026}
Suyoung Choi, \emph{Symplectic small covers in dimension four}, preprint,
  arXiv:2605.01275, 2026.

\bibitem{Choi-Kaji-Theriault2017}
Suyoung Choi, Shizuo Kaji, and Stephen Theriault, \emph{Homotopy decomposition
  of a suspended real toric space}, Bol. Soc. Mat. Mex. (3) \textbf{23} (2017),
  no.~1, 153--161. \MR{3633130}

\bibitem{Choi-Masuda-Oum2017}
Suyoung Choi, Mikiya Masuda, and Sang-il Oum, \emph{Classification of real
  {B}ott manifolds and acyclic digraphs}, Trans. Amer. Math. Soc. \textbf{369}
  (2017), no.~4, 2987--3011. \MR{3592535}

\bibitem{Choi-Masuda-Suh2010Quasitoric}
Suyoung Choi, Mikiya Masuda, and Dong~Youp Suh, \emph{Quasitoric manifolds over
  a product of simplices}, Osaka J. Math. \textbf{47} (2010), no.~1, 109--129.
  \MR{2666127}

\bibitem{Choi-Masuda-Suh2010TAMS}
\bysame, \emph{Topological classification of generalized {B}ott towers}, Trans.
  Amer. Math. Soc. \textbf{362} (2010), no.~2, 1097--1112. \MR{2551516}

\bibitem{Choi-Masuda-Suh2011}
\bysame, \emph{Rigidity problems in toric topology: a survey}, Tr. Mat. Inst.
  Steklova \textbf{275} (2011), 188--201. \MR{2962979}

\bibitem{Choi-Park2017Cohomology}
Suyoung Choi and Hanchul Park, \emph{On the cohomology and their torsion of
  real toric objects}, Forum Math. \textbf{29} (2017), no.~3, 543--553.
  \MR{3641664}

\bibitem{Choi-Park2020}
\bysame, \emph{Multiplicative structure of the cohomology ring of real toric
  spaces}, Homology Homotopy Appl. \textbf{22} (2020), no.~1, 97--115.
  \MR{4027292}

\bibitem{Choi-Park2016IJM}
Suyoung Choi and Seonjeong Park, \emph{Projective bundles over toric surfaces},
  Internat. J. Math. \textbf{27} (2016), no.~4, 1650032, 30. \MR{3491049}

\bibitem{Davis-Januszkiewicz1991}
Michael~W. Davis and Tadeusz Januszkiewicz, \emph{Convex polytopes, {C}oxeter
  orbifolds and torus actions}, Duke Math. J. \textbf{62} (1991), no.~2,
  417--451. \MR{1104531 (92i:52012)}

\bibitem{Huybrechts2005book}
Daniel Huybrechts, \emph{Complex geometry}, Universitext, Springer-Verlag,
  Berlin, 2005, An introduction. \MR{2093043}

\bibitem{Ishida2011}
Hiroaki Ishida, \emph{Symplectic real {B}ott manifolds}, Proc. Amer. Math. Soc.
  \textbf{139} (2011), no.~8, 3009--3014. \MR{2801640}

\bibitem{Kamishima-Masuda2009}
Yoshinobu Kamishima and Mikiya Masuda, \emph{Cohomological rigidity of real
  {B}ott manifolds}, Algebr. Geom. Topol. \textbf{9} (2009), no.~4, 2479--2502.
  \MR{2576506}

\bibitem{Masuda2010}
M.~Masuda, \emph{Cohomological non-rigidity of generalized real {B}ott
  manifolds of height 2}, Tr. Mat. Inst. Steklova \textbf{268} (2010),
  252--257. \MR{2724345}

\bibitem{Masuda-Panov2008}
M.~Masuda and T.~E. Panov, \emph{Semi-free circle actions, {B}ott towers, and
  quasitoric manifolds}, Mat. Sb. \textbf{199} (2008), no.~8, 95--122.
  \MR{2452268}

\bibitem{Nakayama-Nishimura2005}
Hisashi Nakayama and Yasuzo Nishimura, \emph{The orientability of small covers
  and coloring simple polytopes}, Osaka J. Math. \textbf{42} (2005), no.~1,
  243--256. \MR{2132014}

\bibitem{Voisin2007book}
Claire Voisin, \emph{Hodge theory and complex algebraic geometry. {I}},
  {E}nglish ed., Cambridge Studies in Advanced Mathematics, vol.~76, Cambridge
  University Press, Cambridge, 2007, Translated from the French by Leila
  Schneps. \MR{2451566}

\end{thebibliography}

\providecommand{\bysame}{\leavevmode\hbox to3em{\hrulefill}\thinspace}
\providecommand{\MR}{\relax\ifhmode\unskip\space\fi MR }
\providecommand{\MRhref}[2]{%
  \href{http://www.ams.org/mathscinet-getitem?mr=#1}{#2}
}
\providecommand{\href}[2]{#2}

\end{document}